\documentclass[11pt, a4paper, oneside]{amsart}
\pdfoutput=1

\usepackage{amsmath, amssymb, amsthm, bbm, multicol, rotating, tabularx, enumitem}
\usepackage[all]{xy}
\usepackage[a4paper, left=2.5cm, right=3cm, top=3cm, bottom=3cm]{geometry}
\usepackage[english]{babel}

\usepackage{graphicx, color}
\usepackage{ucs}
\usepackage[utf8x]{inputenc}
\usepackage[T1]{fontenc}

\usepackage[pdfpagelabels]{hyperref}

\theoremstyle{plain}
\newtheorem{thm}{Theorem}[section]
\newtheorem{prop}[thm]{Proposition}
\newtheorem{lemma}[thm]{Lemma}
\newtheorem{corollary}[thm]{Corollary}
\newtheorem{property}[thm]{}

\newtheorem*{ThmInf*}{Theorem \ref{Thm_buildingAtInfinity}}

\theoremstyle{definition}
\newtheorem{definition}[thm]{Definition}

\theoremstyle{remark}
\newtheorem{remark}[thm]{Remark}

\numberwithin{equation}{thm}

\newcommand{\R}{\mathbb{R}}

\newcommand{\define}{\mathrel{\mathop:}=}
\newcommand{\ddefine}{\mathrel{=\mathop:}}

\newcommand{\MS}{\mathbb{A}} 
\newcommand{\App}{\mathcal{A}}
\newcommand{\seg}{\mathrm{seg}} 
\newcommand{\binfinity}{\partial_\App} 

\newcommand{\RS}{\Phi}
\newcommand{\sW}{\overline{W}} 
\newcommand{\aW}{W} 
\newcommand{\WT}{{\aW}_T}
\newcommand{\Cf}{\mathcal{{C}}_{f}} 

\begin{document}

\hypersetup{pdfauthor={Petra N. Schwer},pdftitle={Axioms of affine buildings}}.
\title{Axioms of affine buildings}
\author{Petra N. Schwer (n\'ee Hitzelberger)}
\address{Fachbereich Mathematik und Informatik, Universit\"at M\"unster,
Einsteinstrasse~62, 48149 M\"unster, Germany}
\email{hitzelberger@uni-muenster.de}

\thanks{ I would like to thank Koen Struyve for many helpful comments. The author was financially supported by the SFB 478 ``Geometric structure in mathematics'' at the University of M\"unster.} 

\maketitle

\begin{abstract}
\vspace{1ex}
We prove equivalence of certain axiom sets for affine buildings. 
Along the lines a purely combinatorial proof of the existence of a spherical building at infinity is given. As a corollary we obtain that ``being an affine building'' is independent of the metric structure of the space.
\end{abstract}

\section{Introduction}
\label{sec_introduction}

Verifying that an object satisfies a certain list of axioms can sometimes be a problem hard to tackle. Once in a while one might wish that there is a shorter equivalent axiom set suited better for the purpose of a given problem. While working on a different project \cite{BaseChange} we had to verify that a certain space is a generalized affine building in the sense of Bennett \cite{Bennett}. This was the motivation to prove the main result of the present paper, that is Theorem~\ref{MainThm}. 

Adding one axiom to Tits' list defining non-discrete $\R$ buildings, Bennett was able to generalize the concept to arbitrary ordered abelian groups. In \cite{Bennett} and \cite{BennettDiss} he defined generalized affine buildings giving a list of six axioms. Later, for their proof of the Margulis conjecture in \cite{KramerTent}, Kramer and Tent made use of the theory of generalized affine buildings. Recently they have been studied by the author in \cite{Diss} and \cite{Convexity2}.

Our purpose is to study equivalent sets of axioms for generalized affine buildings. We will reduce the number of axioms and obtain that a universal definition for both $\R$-buildings and affine buildings defined over arbitrary Krull-valuated fields can be given.
From our main result we deduce that the building structure does not depend on its metric. In other words, whichever metric one might impose on the model apartment, the induced distance function on the affine building will be a metric. In particular does the induced metric always satisfy the triangle inequality.

In \cite{Bennett, Brown} or \cite{Parreau} the triangle inequality is solely used to prove existence of a spherical building at infinity. We were able to find an equivalent definition of parallelism of Weyl simplices which is  purely combinatorial and does not build on the metric structure of the affine building.
This helps us to prove that the building structure does not depend on the metric imposed on the apartment level.

Equivalent sets of axioms for affine $\R$- buildings have been previously studied by Anne Parreau \cite{Parreau}. This paper extends her results. 
Further did the author have access to a preprint by Curt Bennett \cite{Bennett2} which is also devoted to a reduction of the axioms of a generalized affine buildings. He did replace the difficult to verify (A6) by easier alternatives.

The original axiomatic definition of affine buildings is due to Jaques Tits. He defined the ``syst{\`e}me d'appartements'' in \cite{TitsComo} by listing five axioms. The first four of these are precisely axioms $(A1)-(A4)$ as presented in the following section. His fifth axiom originally reads different from ours but was later replaced with what is now axiom (A5) in Definition~\ref{Def_LambdaBuilding}.
The interested reader can find a short history of Tits' axioms in Marc Roman's book \cite{Ronan}. 
As already mentioned above in 1994 Bennett introduced the notion of a generalized affine building, by adding an additional axiom to Tits' list. He gave an example showing that the new axiom (A6) might not be omitted.

Assuming that the metric induced by the Euclidean distance on one apartment satisfies the triangle inequality, 
Anne Parreau later proved equivalence of (A5) and (A6) in case $\Lambda=\R$. 
In her proof the triangle inequality is needed to show the existence of the spherical building at infinity. In fact each known proof of the existence of the spherical building at infinity uses, in one way or another, the retraction appearing in axiom (A5) or the triangle inequality for the distance function on the building $X$, which is proved using (A5).

Axiom (A5) being equivalent to (A6) plus triangle inequality  in case $\Lambda=\R$ suggest that we should find a  purely combinatorial proof of the existence of the building at infinity in oder to obtain that axiom (A5) is superfluous in Definition~\ref{Def_LambdaBuilding}. This is carried out in Section~\ref{Sec_infinity}.

Besides the alternative proof of the existence of a building at infinity we will, in this short note, mainly discuss alternative sets of axioms for generalized affine buildings. In the following subsection we define generalized affine buildings and list the properties in consideration. 
For details we refer the reader to \cite{Diss} and \cite{Bennett}.

\subsection{Equivalence of axioms}

The model apartment of a generalized affine building  is defined by means of a (not necessarily crystallographic) spherical root system $\RS$ and a  totally ordered abelian group $\Lambda$. As the apartments of Euclidean buildings are isomorphic copies of $\R^n$ so is the \emph{model space} $\MS$ of a generalized affine building isomorphic to $\Lambda^n$.
We define 
$$
 \MS(\RS,\Lambda) = \mathrm{span}_F(\RS)\otimes_F \Lambda,
$$ 
where $F$ is a sub-field of the reals containing all evaluations of co-roots on roots. 

The spherical Weyl group $\sW$ associated to $\RS$ acts on $\MS$. A \emph{hyperplane} $H_\alpha$ in the model space is a fixed point set of a reflection $r_\alpha$ in $\sW$ which separates $\MS$ into two half-spaces, called \emph{half-apartments}.
There is as well an affine Weyl group $\WT$ acting on $\MS$, which is the semi-direct product of $\sW$ by some $\sW$ invariant translation group $T$ of the model space. In case the translation group $T$ is the entire space $\MS$ we write $\aW$ instead of $\WT$. 

Associated to a basis $B$ of the root system $\RS$ there is a \emph{fundamental Weyl chamber} $\Cf$. The chamber $\Cf$ is a fundamental domain for the action of $\sW$ on $\MS$ and its images under the affine Weyl group are the \emph{Weyl chambers} in $\MS$. A \emph{Weyl simplex} is a face of a Weyl chamber. The smallest face of dimension 0 is called \emph{basepoint}.

One can endow $\MS$ with a natural $\aW$-invariant metric taking its values in $\Lambda$ and making $\MS$ a $\Lambda$-metric space in the sense of Definition \ref{Def_metric}. 

\begin{definition}\label{Def_LambdaBuilding}
Let $X$ be a set and $\App$ a collection of injective charts $f:\MS\hookrightarrow X$.
We call the images $f(\MS)$  of the charts $f$ in $\App$ \emph{apartments} of $X$ and we define \emph{Weyl simplices, hyperplanes, half-apartments, ... of $X$} to be images of such in $\MS$ under a chart in $\App$. The set $X$ is a \emph{(generalized) affine building} with \emph{atlas} $\App$ if the following conditions are satisfied
\begin{enumerate}[label={(A*)}, leftmargin=*]
\item[(A1)] The atlas is invariant under pre-composition with elements of  $ \WT$.
\item[(A2)] Given two charts $f,g\in\App$ with $f(\MS)\cap g(\MS)\neq\emptyset$. Then $f^{-1}(g(\MS))$ is a closed convex subset of $\MS$ and there exists $w\in \WT$ with $f\vert_{f^{-1}(g(\MS))} = (g\circ w )\vert_{f^{-1}(g(\MS))}$.
\item[(A3)] For any pair of points  in $X$ there is an apartment containing both.
\end{enumerate}
Given a $\Lambda$-metric on the model space,  axioms $(A1)-(A3)$ imply the existence of a $\Lambda$-valued distance on $X$, that is a function $d:X\times X\mapsto \Lambda$ satisfying all conditions of the definition in \ref{Def_metric} but the triangle inequality.  The distance of points $x,y$ in $X$ is the distance of their preimages under a chart $f$ of an apartment containing both.
\begin{enumerate}[label={(A*)}, leftmargin=*]
\item[(A4)] Given two  Weyl chambers  in $X$ there exist sub-Weyl chambers of both which are contained in a common apartment. 
\item[(A5)] For any apartment $A$ and all $x\in A$ there exists a \emph{retraction} $r_{A,x}:X\to A$ such that $r_{A,x}$ does not increase distances and $r^{-1}_{A,x}(x)=\{x\}$.
\item[(A6)] Let $f, g$ and $h$ be charts such that the associated apartments pairwise intersect in half-apartments. Then $f(\MS)\cap g(\MS)\cap h(\MS)\neq \emptyset$. 
\end{enumerate}
By $(A5)$ the distance function $d$ on $X$ is well defined and satisfies the triangle inequality.
\end{definition}

The main goal of the present paper is to prove equivalence of certain sets of axioms. Let us therefore collect all properties which are necessary to state the main result.

\begin{itemize}
 \item[(EC)] Given two apartments $A$ and $B$ intersecting in a half-apartment $M$ with boundary wall $H$, then $(A\bigoplus B)\cup H$ is also an apartment, where $\bigoplus$ denotes the symmetric difference. 
\end{itemize}

We say that two Weyl simplices $S$ and $T$ \emph{share the same germ} if both are based at the same vertex and if $S\cap T$ is a neighborhood of $x$ in $S$ and in $T$.
It is easy to see that this is an equivalence relation on the set of Weyl simplices based at a given vertex. The equivalence class of $S$, based at $x$, is denoted by $\Delta_x S$ and is called the \emph{germ of $S$ at $x$}.

A germ $\mu$ of a Weyl chamber $S$ at $x$ is \emph{contained} in a set $Y$ if there exists $\varepsilon\in\Lambda^{+}$ such that $S\cap B_\varepsilon(x)$ is contained in $Y$.

\begin{itemize}
 \item[(A3')] Any two germs of Weyl chambers are contained in a common apartment. 
 \item[(A3'')] For all points $x$ and $y$-based Weyl chambers $S$ there exists an apartment containing both $x$ and $\Delta_yS$.
 \item[(GG)] Any two germs of Weyl chambers based at the same vertex are contained in a common apartment.
\end{itemize}

We will be able to prove that under certain assumptions the set $\Delta_xX$ of all germs of Weyl simplices at a fixed point $x$ in $X$ carries the structure of a spherical building. The germs of Weyl chambers will be the chambers in $\Delta_xX$.
We say that two germs of Weyl chambers are \emph{ opposite at $x$} if they are opposite as chambers in the building $\Delta_xX$. 

\begin{itemize}
 \item[(CO)] Two Weyl chambers $S$ and $T$,  which are based at the same vertex $x$ and whose germs are opposite at $x$,  are contained in a unique common apartment.
\end{itemize}

The segment $\seg(x,y)$ of points $x$ and $y$ in a metric space $X$ is the set of points $z$ such that $d(x,y)=d(x,z)+d(z,y)$. Let $A$ be an apartment in an affine building  containing two points $x$ and $y$. We write $\seg_A(x,y)$ for the intersection of $\seg(x,y)$ with $A$. 
 
\begin{itemize}
 \item[(FC'')] For all triples of points $x,y$ and $z$ in $X$ and all apartments $A$ containing $x$ and $y$ the segment $\seg_A(x,y)$ is contained in a finite union of Weyl chambers based at $z$.
\end{itemize}

\begin{remark}
Property (EC) was introduced by Bennett \cite{Bennett2} as an alternative to the sixth axiom.
Axiom (A3') is a stronger version of (A3) and the precise analog of the simplicial condition that two (affine) chambers are always contained in a common apartment. Both, (A3') and property  (GG), were introduced by Parreau \cite{Parreau}. 
Property (CO) did as well appear in \cite{Parreau} first.
Axiom (A3'') is 'in between' (A3) and (A3') and suffices for one of the implications in \ref{MainThm}. 
In \cite{Convexity2} we used a slightly stronger version of property (FC'') to prove that certain retractions are distance diminishing. However, Koen Struyve noticed that (FC'') suffices for our purposes.
\end{remark}

We say that $(X,\App)$ is \emph{a space modeled on $\MS$} if $X$ is a set together with a collection $\App$ of injective \emph{charts} $f:\MS\hookrightarrow X$ such that $X$ is covered by its charts. That is $X=\bigcup_{f\in\App} f(\MS)$.

\begin{thm}\label{MainThm}
For a space $(X,\App)$ modeled on $\MS=\MS(\RS,\Lambda)$ which satisfies axioms (A1)-(A3), the following are equivalent:
\begin{enumerate}
 \item\label{i} $(X,\App)$ is a generalized affine building, that is axioms (A4), (A5) and (A6) are satisfied.
 \item\label{ii} Axioms (A4), (A5) and (EC) hold.
 \item\label{iii} Axioms (A4) and (A6) are satisfied.
 \item\label{iv} Properties (GG) and (CO) hold.
 \item\label{v} The pair $(X,\App)$ has properties (A3') and (CO).
 \item\label{vi} Axioms (A3''), (A4) and properties (FC'') and (EC) are satisfied.
\end{enumerate}
\end{thm}
Obviously if one of the properties (A3') and (A3'') hold axiom (A3) is superfluous. 

We will prove the following implications:
$$
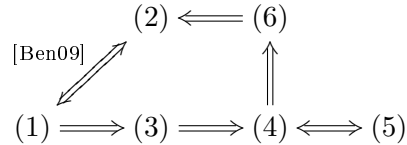
\xymatrix{
	& (2) \ar@{<=>}[dl]_{\text{\cite{Bennett2}}} 	& (6) \ar@{=>}[l]\ar@{<=}[d] & \\
(1) \ar@{=>}[r]	& (3)	& (4) \ar@{<=}[l]\ar@{<=>}[r] &  (5) 
}
$$
\vspace{2ex}

The fact that (A6) and (EC) are equivalent assuming (A1) to (A5) is due to Bennett \cite{Bennett2}.
We obtain (GG) and (CO) as discussed in Section~\ref{Sec_localStructure} (compare Corollaries \ref{Cor_GG} and \ref{Cor_CO}). Hence item (\ref{iii}) implies (\ref{iv}).

Section \ref{Sec_A3'} contains the proof of the fact that (\ref{iv}) implies property (A3') and hence (A3''). Later, in Section~\ref{Sec_A4} axiom (A4) is shown assuming (\ref{iv}). The exchange condition (EC) holds as outlined  Section~\ref{Sec_A6'}. Finally, as shown in Section~\ref{Sec_FC1}, condition (FC'') follows from (A1) to (A3) and (CO). This completes the proof of the fact that (\ref{iv}) implies (\ref{vi}).

Axiom (A5) is verified in Section~\ref{Sec_A5} using (A1), (A2), (A3'') and (FC''). Therefore item (\ref{vi}) implies (\ref{ii}).
Compare Section~\ref{Sec_A3'} and \ref{Sec_lp} to obtain that the axioms listed in (\ref{iv}) are equivalent to the ones in (\ref{v}).  See \ref{Sec_A3'} for the fact that  (\ref{iv}) implies (\ref{v}).
The converse, that (\ref{v}) implies (\ref{iv}), is proved in Section~\ref{Sec_lp}.

\subsection{Further results}

Let me start this section with a simple yet interesting consequence of Theorem~\ref{MainThm}.
The class of generalized affine building is a generalization of $\R$-buildings, which themselves generalize the (geometric realizations of) simplicial affine buildings. The $\R$-buildings are the sub-class where $\Lambda=\R$ and where the translational part $T$ of the affine Weyl group equals the co-root-lattice spanned by a crystallographic root system, or is the full translation group of an apartment in the non-crystallographic case. 

For this we are using the metric approach to affine buildings, replacing $\R$-metric spaces by $\Lambda$-metric spaces in the following sense. 

\begin{definition}\label{Def_metric}
A  \emph{$\Lambda$-metric} on a space $X$, is a map $d:X\times X \mapsto \Lambda$ such that for all $x,y,z$ in $X$ the following axioms are satisfied
\begin{enumerate}
\item $d(x,y)=0$ if and only if $x=y$
\item $d(x,y)=d(y,x)$ and
\item the triangle inequality $d(x,z)+d(z,y)\geq d(x,y)$ holds.
\end{enumerate}
\end{definition}

There is however a small problem in viewing Euclidean buildings as a subclass of affine buildings. The definition of an affine building is based on the definition of a given model space, which in turn comes with a fixed metric. In case of $\R$-buildings one usually uses the Euclidean metric on the model space. Therefore the metric on the affine building $X$ is, when restricted to an apartment, precisely the Euclidean metric. Compare for example \cite{Parreau} or Kleiner and Leeb \cite{KleinerLeeb}. 

The natural metric on the model space of a generalized affine building is however defined in terms of the defining root system $\RS$, compare \cite{Diss}.
It is a generalization of the length of translations in apartments of simplicial affine buildings. This length function on the set of translational elements of the affine Weyl group is defined with respect to the length of certain minimal galleries. 
The problem is that this natural metric used for ``$\Lambda$-buildings'' is different from the Euclidean one in case $\Lambda=\R$.
For our purposes it is not necessary to specify any details. We simply assume throughout the following that there exists some  $\aW$-invariant $\Lambda$-metric on $\MS$. 

The question arising is the following: Let us assume that $X$ is an affine building with metric $d$, which is induced by a metric $d_\MS$ on the model space. Let $d'_\MS$ be a metric on the model space, which differs from $d_\MS$. Hence $d'_\MS$ induces a second distance function $d'$ on $X$. 
Does $d'$ satisfy the triangle inequality? And is $(X,d')$ an affine building? To be able to answer these questions one has to understand whether the retractions appearing in (A5) do exist and are distance diminishing. The answer to these questions is ``yes'', and using Theorem \ref{MainThm} we do not need  to prove (A5) directly. 

\begin{corollary}
Let $(X,\App)$ be an affine building. Then every metric on the model space extends to a metric on $X$.
\end{corollary}
\begin{proof}
Since $(X,\App)$ is a building axioms (A6) and (A1) to (A4) are satisfied. These axioms do not contain conditions on the metric and are, by \ref{MainThm} equivalent to the ones listed in Definition~\ref{Def_LambdaBuilding}. Hence every distance function on $X$ which is induced by a metric on the model space satisfies the triangle inequality.
\end{proof}

Thus whether or not a pair $(X,\App)$ modeled on $\MS$ is an affine building does not depend on the metric imposed on $\MS$.
This consequence of our main result makes use of the fact that (A5) can be omitted in Definition \ref{Def_LambdaBuilding}.

The basic idea is to find a purely combinatorial definition of parallelism of Weyl simplices which allows us to prove existence of a spherical building at infinity without using the metric structure of the affine building. Finally this enables us to eliminate axiom (A5) in the definition of an affine building.

Bennett \cite{Bennett} did prove already that two Weyl chambers, which are contained in the same apartment, are at bounded distance if and only if they are translates of one another.
Using this one observes that ``being at bounded distance in the building'' is the same as being, in a certain sense, ``translates of one another''.

This new approach makes the definition of parallelism a bit lengthy but avoids using the metric. 
The details are carried out in Section~\ref{Sec_infinity}, where we prove the following theorem.

\begin{ThmInf*}
Let $(X,\App)$ be a pair satisfying axioms (A1)-(A4). Then
$$
\binfinity X \define\{\partial F : F \text{ is a Weyl simplex in } X\}
$$ 
is a spherical building of type $\RS$ with apartments in one to one correspondence with the apartments of $X$.
\end{ThmInf*}

\subsection*{The remainder of the present paper is paper is organized as follows.} 
In the next section we will give a combinatorial definition of parallelism of Weyl simplices. Using this we prove the existence of the spherical building at infinity using axioms (A1) to (A4), only. 

The rest of the paper, Sections \ref{Sec_infinity} to \ref{Sec_A6'}, need not be read sequentially. There we prove one after another the implications of  \ref{MainThm} as shown in  the diagram on page \pageref{MainThm}. The only sections which are better read in a row are Sections~\ref{Sec_retractions} to \ref{Sec_A5}.
Otherwise the best possible strategy might be to pick ones favorite inclusion and read the sections needed for its proof. We did already say, after stating the main theorem in the previous subsection,  where to find what.

\section{The building at infinity}\label{Sec_infinity}

Any simplicial affine building has an associated spherical building at infinity. 
Most of the constructions of the building at infinity found in the literature, such as the one in \cite{Parreau} or \cite{Brown, AB} for example, heavily rely on the metric structure of the affine building.
Bennett's \cite{Bennett} proof for generalized affine buildings did rely on metric properties as well. 

The purpose of the present section is to provide a definition of parallelism for Weyl simplices that does not involve the metric structure of the affine building and which allows a new, combinatorial proof for the existence of a spherical building at infinity.
To be precise, in comparison to \cite{Bennett}, we avoid using axiom (A5) in the proof.

\begin{definition}\label{Def_buildingAtInfinity}\label{Def_parallelSimplices}
\index{building at infinity}
Let $(X,\App)$ be a pair satisfying axioms (A1)-(A4). We say that $S$ and $T$ are \emph{parallel} if $S\cap T$ contains a Weyl chamber. We denote by $\partial S$ the parallel class of $S$. 

As we will see later on in this section, the set 
$$
\{\partial S : S \text{ Weyl chamber of } X \text{ contained in an apartment of } \App\}
$$
of equivalence classes of Weyl chambers is the collection of chambers of a spherical building \emph{at infinity of $X$}.
\end{definition}

Bennett defined two Weyl simplices to be parallel if they are at bounded Hausdorff distance. One can proof, compare \cite[4.23]{Diss} and \cite{Bennett}, that ``being at bounded distance'' can be characterized differently. 

\begin{prop}\label{Prop_parallel}
Given two Weyl chambers $S$ and $T$ the following are equivalent 
\begin{enumerate}
 \item They are parallel in the sense of Definition~\ref{Def_parallelSimplices}. 
 \item They contain sub-Weyl chambers $S'\subset S$ and $T'\subset T$ such that $S'$ and $T'$ are contained in a common apartment and are translates of one another in this apartment.
 \item $S$ and $T$ are at bounded Hausdorff distance, i.e. are parallel in the sense of \cite{Parreau}.
\end{enumerate}
\end{prop}

\begin{lemma}\label{Lem_subWeyl}
\begin{enumerate}
\item If $C$ is a sub-Weyl chamber of $D$, then $C$ is a translate of $D$. 
\item If $S$ is a translate of the Weyl chamber $T$ in an apartment $A$, then $S\cap T$ contains a common Weyl chamber of both.
\item Given sub-Weyl chambers $S$ and $T$ of the same Weyl chamber $U$, then $S\cap T$ contains a Weyl chamber.
\end{enumerate}
\end{lemma}
\begin{proof}
Since $C$ is a sub-Weyl chamber of $D$ these two are at bounded distance. By Proposition 4.23.2 in \cite{Diss} there exists then a Weyl chamber $U\subset C\cap D = C$ having bounded distance to both. By 1. of the same Proposition this is equivalent to the fact that one is a translate of the other. Hence (1).

To prove the second assertion observe that $S=t+T$ for some translation $T$ in the affine Weyl group. Therefore $S$ ant $T$ are parallel in the sense of Bennett by \cite[Prop. 2.7]{Bennett}. Proposition  \cite[Proposition 3.4]{Bennett} implies that $S\cap T$  contains a sub-Weyl chamber parallel to both. This implies (2).

Using the first item we can conclude that the sub-Weyl chambers $S$ and $T$ of $U$ in the last item are both translates of $U$. Hence $S$ is a translate of $T$ and they are, by 2.23.1 in \cite{Diss}, at bounded distance of one another.  By the second assertion of the same proposition, their intersection therefore contains a sub-Weyl chamber of both. 
\end{proof}

Let $F$ and $G$ be Weyl simplices in an affine building $X$. Let $S$ and $T$ be Weyl chambers such that $F$ is a face of $S$ and $G$ one of $T$. By (A4) there exists an apartment $A$ containing sub-Weyl chambers $S'\subset S$ and $T'\subset T$. In an apartment containing $S$ the sub-Weyl chamber $S'$ is a translate of $S$ and thus there exists a face $F'$ of $S'$ which is a translate of $F$ in this apartment. We say that $F'$ \emph{corresponds to} $F$. In the same manner there is a face $G'$ of $T'$ corresponding to $G$.

\begin{definition}\label{Def_parallelSimplex}
Two Weyl simplices $F$ and $G$ are \emph{parallel} if the corresponding Weyl simplices $F'$ and $G'$ we described above are translates of one another in an apartment containing both. 
\end{definition}

This definition is clearly independent of the choice of $A$ since every sub-Weyl chamber of a Weyl chamber $S$ is a translate of $S$ in every apartment containing $S$ (see Lemma~\ref{Lem_subWeyl}). By Proposition~\ref{Prop_parallel} it is equivalent to the definition used in \cite{Bennett} or \cite{Parreau}. 

\begin{prop}\label{Prop_parallelEquiv}
Parallelism is an equivalence relation on Weyl simplices. 
\end{prop}
\begin{proof}
Reflexivity and symmetry is clear. Hence it remains to prove transitivity. Let $F$, $G$ and $H$ be Weyl simplices such that $G$ is parallel to $F$ and parallel to $H$. It is to prove that $F$ is parallel to $H$ as well. 
Since $F$ and $G$ are parallel there exist translates $F'$ and $G'$, respectively, which are contained in a common apartment $A$ in which they are translates of one another. Thus there exists a translation $t\in \aW$ such that $F'= t+G'$. For the same reason there exists an apartment $B$ containing translates $G''$ and $H''$ of $G$, respectively $H$. Furthermore there is a translation $s$ such that $G''= s+ H''$.

For the following reason we may assume that $G'=G''$: To find $A$ and $B$ we need to apply Definition~\ref{Def_parallelSimplex} to the pairs $F$,$G$ and $G$, $H$. We may use in both cases the same Weyl chamber $S$ having $G$ as a face. Doing so we obtain sub-Weyl chambers $S'$ in $A$ and $S''$ in $B$ having $G'$, respectively $G''$ as a face. Replacing, if necessary,  $S'$ and $S''$ by a common sub-Weyl chamber $S'''$ we may assume that $S'=S''$ and that $G'=G''$. 

Hence we are in the following situation.
The Weyl simplex $F'$ is a face of the Weyl chamber $T'$ which is contained in the same apartment $A$ as the Weyl chamber $S'$ which has $G'$ as a face. Furthermore $F' = t + G'$ in $A$ and $G'=s+H'$ in $B$. The Weyl simplex $H'$ is a face of $U'$, a Weyl chamber contained in $B$ which is an apartment containing $S'$. In particular $S'$ is contained in the intersection of $A$ and $B$. 

The translate $C\define t+S'$ of $S'$ is also a Weyl chamber in $A$ having $F'$ as a face and $D\define -s+ S'$ is a Weyl chamber in $B$ with face $H'$. The intersection of $D$ and $S'$ contains a Weyl chamber $D'$ and the intersection $S' \cap C$ contains a Weyl chamber $C'$    
Both, $C'$ and $D'$, are sub-Weyl chambers of $S'$. By Lemma~\ref{Lem_subWeyl} their intersection thus contains a Weyl chamber $C''$.

By the arguments above $C'$ is a translate of $D'$ in every  apartment which contains $S'$. 
The face $F'$ is parallel to $G'$ and the Weyl simplex $G'$ is parallel to $H'$. Therefore $F'$ is a translate of $F''\subset C'$ and $H'$ is a translate of  $H''\subset D'$. This implies that $F''$ is a translate of $H''$. 
Hence $F'$ is parallel to $H'$ in the sense of Definition~\ref{Def_parallelSimplex} 
\end{proof}

We say that $\partial F$ \emph{is a face of} $\partial S$ if there exist representatives $F$ and $S$ such that $F$ is a face of $S$. This defines a simplicial structure on parallel classes of Weyl simplices.
We define two parallel classes $\partial F$ and $\partial G$ of Weyl simplices to be \emph{ adjacent} if there exist representatives based at the same vertex and having a face in common. 

\begin{thm}\label{Thm_buildingAtInfinity}
Let $(X,\App)$ be a pair modeled on $\MS(\RS, \Lambda, T)$ satisfying axioms (A1)-(A4). Then the set 
$$
\binfinity X \define\{\partial F : F \text{ is a Weyl simplex in } X\}
$$ 
is a spherical building of type $\RS$ with apartments in one to one correspondence with the apartments of $X$.
\end{thm}
\begin{proof}
By definition of adjacency  the set $\binfinity X$ is a chamber complex. The sub-complex consisting of all equivalence classes of Weyl simplices contained in a fixed apartment is isomorphic to a Coxeter complex of type $\RS$ if $X$ is modeled on the root system $\RS$. These sub-complexes are the apartments of $\binfinity X$. Axiom (A4) implies that two chambers $\partial S$ and $\partial T$ are contained in a common apartment. Following \cite[p.76/77]{Brown} it remains to prove that two apartment of $\binfinity X$ which contain a common chamber are  isomorphic via an isomorphism fixing their intersection, that is (B2'').

Let $A, A'$ be apartments and $c$ a chamber in $\partial A\cap \partial A'$. Then there exist representatives $S\subset A$ and $S'$ in $A'$ of the equivalence class $c$. Hence $S\cap S'$  contains a sub-Weyl chamber $S''$. 
Therefore we can find charts $f,f'$ of $A,A'$ such that 
$$
f'\circ f^{-1}\vert_{A\cap A'} = id\vert_{A\cap A'}.
$$
The induced map $\partial(f'\circ f^{-1})$ at infinity is an isomorphism fixing $\partial A \cap \partial A'$.
\end{proof}

\section{Local structure}\label{Sec_localStructure}

Let in the following $(X,\App)$ be a pair modeled on $\MS=\MS(\Lambda, \RS, T)$ and satisfying all axioms but (A5). Recall from the previous section that this is enough to conclude that $\binfinity X$ is a spherical building. 

\begin{definition}\label{Def_germ}
\index{{generalized affine building}!{germ}}
Two Weyl simplices $S$ and $S'$ \emph{share the same germ} if both are based at the same vertex and if $S\cap S'$ is a neighborhood of $x$ in $S$ and in $S'$.
\end{definition}
It is easy to see that this is an equivalence relation on the set of Weyl simplices based at a given vertex. The equivalence class of an $x$-based Weyl simplex $S$ is denoted by $\Delta_x S$ and is called the \emph{germ of $S$ at $x$}.

The germs of Weyl simplices at a special vertex $x$ are partially ordered by inclusion: $\Delta_x S_1$ is contained in $\Delta_xS_2$ if there exist $x$-based representatives $S'_1, S'_2$ contained in a common apartment such that $S_1'$ is a face of $S_2'$. Let $\Delta_xX$ be the set of all germs of Weyl simplices based at $x$.

Recall that a germ $\mu$ of a Weyl chamber $S$ at $x$ is \emph{contained in a set $Y$} if there exists $\varepsilon\in\Lambda^{+}$ such that $S\cap B_\varepsilon(x)$ is contained in $Y$.

\begin{prop}\label{Prop_tec16}
Let $(X, \App)$ be an affine building and $c$ a chamber in $\binfinity X$. Let $S$ be a Weyl chamber in $X$ based at $x$. Then there exists an apartment $A$ such that $\Delta_xS$ is contained in  $A$ and  such that $c$ is a chamber of $\partial A$.
\end{prop}
The proof of the proposition above is precisely the same as the proof of Proposition 1.8 in \cite{Parreau}. 
Parreau's proof uses the fact that $\binfinity X$ is a spherical building and that axioms (A1) to (A3) as well as  (A6) are satisfied. Recall that assuming (A1) to (A4) we where able to prove in Section \ref{Sec_infinity} that $\binfinity X$ is a spherical building.

\begin{corollary}\label{Cor_GG}
Any pair $(X,\App)$ satisfying all axioms but (A5) has property (GG).
\end{corollary}
\begin{proof}
Let $S$ and $T$ be Weyl chambers both based at a point $x$. By Proposition~\ref{Prop_tec16} there exists an apartment $A$ of $X$ containing $S$ and a germ of $T$ at $x$.
\end{proof}

Notice that, by the previous corollary, such a pair $(X,\App)$ satisfies the assertion of Theorem~\ref{Thm_residue}, i.e. the germs at a fixed vertex form a spherical building. Hence the notion of opposite germs as defined in the introduction makes sense.

\begin{prop}\label{Prop_A3'}
If $(X,\App)$ is a pair satisfying all axioms but (A5) then property (A3') holds.
\end{prop}
\begin{proof}
We need to prove that if $S$ and $T$ are Weyl chambers based at $x$ and $y$, respectively, then there exists an apartment containing a germ of $S$ at $x$ and a germ of $T$ at $y$.

By axiom $(A3)$ there exists an apartment $A$ containing $x$ and $y$. We choose an $x$-based Weyl chamber $S_{xy}$ in $A$ that contains $y$ and denote by $S_{yx}$ the Weyl chamber based at $y$ such that $\partial S_{xy}$ and $\partial S_{yx}$ are v in $\partial A$. Then $x$ is contained in $S_{yx}$. If $\Delta_yT$ is not contained in $A$ apply Proposition~\ref{Prop_tec16} to obtain an apartment $A'$ containing a germ of $T$ at $y$ and containing  $\partial S_{yx}$ at infinity. But then $x$ is also contained in $A'$. 

Let us denote by $S'_{xy}$ the unique Weyl chamber contained in $A'$ having the same germ as $S_{xy}$ at $x$.
Without loss of generality we may assume that the germ $\Delta_yT$ is contained in $S'_{xy}$. Otherwise $y$ is contained in a face of $S'_{xy}$ and we can replace $S'_{xy}$ by an adjacent Weyl chamber in $A'$ satisfying this condition.
A second application of Proposition~\ref{Prop_tec16} to $\partial S'_{xy}$ and the germ of $S$ at $x$ yields an apartment $A''$ containing $\Delta_xS$ and $S'_{xy}$ and therefore $\Delta_yT$.
\end{proof}

Propositions \ref{Prop_A5} to \ref{Prop_liftGallery} below  are due to Linus Kramer.

\begin{prop}\label{Prop_A5}
With $X$ as above let $A_i$ with $i=1,2,3$ be three apartments of $X$ pairwise intersecting in half-apartments. Then $A_1\cap A_2\cap A_3$ is either a half-apartment or a hyperplane.
\end{prop}
The proof of this proposition, which can be found in \cite{Diss}, uses the fact that $\binfinity X$ is a spherical building, hence (A1)-(A4) and axiom (A6).

\begin{property}{\bf The sundial configuration. } \label{Prop_sundial}
Let $A$ be an apartment in $X$ and let $c$ be a chamber not contained in $\partial A$ but containing a panel of $\partial A$. Then $c$ is opposite to two uniquely determined chambers $d_1$ and $d_2$ in $\partial A$. Hence there exist apartments $A_1$ and $A_2$ of $X$ such that $\partial A_i$ contains $d_i$ and $c$ with $i=1,2$.  The three apartments $\partial A_1,\partial A_2$ and $\partial A$ pairwise intersect in half-apartments. Axiom $(A6)$ together with the proposition above implies that their intersection is a hyperplane. 
\end{property}

\begin{prop}\label{Prop_liftGallery}
Let $x$ be an element of $X$. Let $(c_0, \ldots, c_k)$ be a minimal gallery in $\binfinity X$. We denote by $S_i$ the $x$-based representative  of $c_i$. If $(\pi_x(c_0), \ldots, \pi_x(c_k))$ is minimal in $\Delta_xX$, then there exists an apartment containing $\bigcup_{i=0}^k S_i$.
\end{prop}
In \cite{Diss} Proposition~\ref{Prop_liftGallery} is proved by induction on $k$ using the sundial configuration. 

\begin{corollary}\label{Cor_CO}
Every pair $(X,\App)$ satisfying all axioms but (A5) has the property (CO).
\end{corollary}
\begin{proof}
Choose a minimal gallery $(c_0,c_1,\ldots, c_n)$ from $c_0=\partial S$ to $c_n=\partial T$ and consider the representatives $S_i$ of $c_i$ based at $x$. Then $S_0=S$ and $S_n=T$ and Proposition~\ref{Prop_liftGallery} implies the assertion.
\end{proof}

\section{Property (A3')}\label{Sec_A3'}

Assume that $(X,\App)$ is a pair satisfying axioms (A1) to (A3) and properties (GG) and (CO).
By axiom (A2) we may observe that the apartment in property (CO) is unique. 

\begin{thm}\label{Thm_residue}
\index{{generalized affine building}!{residue}}
Assume that $(X,\App)$ is a pair satisfying axioms (A1) to (A3) and property (GG).
Then $\Delta_xX$ is a spherical building of type $\RS$ for all $x$ in $X$. Furthermore $\Delta_xX$ is independent of $\App$.
\end{thm}
\begin{proof}
We verify the axioms of the definition of a simplicial building, which can be found on  page 76 in \cite{Brown}.
It is easy to see that $\Delta_xX$ is a simplicial complex with the partial order defined above. It is a pure simplicial complex, since each germ of a face is contained in a germ of a Weyl chamber. The set of equivalence classes determined by a given apartment of $X$ containing $x$ is a subcomplex of $\Delta_xX$ which is, obviously, a Coxeter complex of type $\RS$. Hence we define those to be the apartments of $\Delta_xX$. Therefore, by definition, each apartment is a Coxeter complex. 
Two apartments of $\Delta_xX$ are isomorphic via an isomorphism fixing the intersection of the corresponding apartments of $X$, hence fixing the intersection of the apartments of $\Delta_xX$ as well. Finally due to property (GG) any two chambers are contained in a common apartment and we can conclude that $\Delta_xX$ is a spherical building of type $\RS$.

Let $\App'$ be a different system of apartments of $X$ and assume w.l.o.g.~that $\App\subset \App'$.
We will denote by $\Delta$ the spherical building of germs at $x$ with respect to $\App$ and  by $\Delta'$ the building at $x$ with respect to $\App'$. Since spherical buildings have a unique apartment system $\Delta$ and $\Delta'$ are equal if they contain the same chambers. Assume there exists a chamber $c\in\Delta'$ which is not contained in $\Delta$. Let $d$ be a chamber opposite $c$ in $\Delta'$ and $a'$ the unique apartment containing both. Note that $a'$ corresponds to an apartment $A'$ of $X$ having a chart in $\App'$. There exist $\App'$-Weyl chambers $S_c$, $S_d$ contained in $A$ representing $c$ and $d$, respectively. Choose a point $y$ in the interior of $S_c$ and let $z$ be contained in the interior of $S_d$. By axiom $(A3)$ there exists a chart $f\in \App$ such that the image $A$ of $f$ contains $y$ and $z$. Then $x$ is contained in $A$ as well, since  $x$ is contained in $\seg_A(y,z)\define\seg(y,z)\cap A$ and the segment $\seg_A(y,z)$ is a subset of $A\cap A'$. By construction the unique $x$-based Weyl chamber in $A$ which contains $y$ has germ $x$ and the unique $x$-based Weyl chamber in $A$ containing  $z$ has germ $d$. This contradicts the assumption that $c$ isn't contained in $\Delta$. Hence $\Delta=\Delta'$.
\end{proof}

As in \cite[Prop 1.15]{Parreau} we observe:
\begin{lemma}\label{lem_tec1}
Let $S$ and $T$ be two $x$-based Weyl chambers. Then there exists an apartment containing $S$ and a germ of $T$ at $x$.  
\end{lemma}

\begin{prop}
 Under the hypotheses of this section we have
\begin{itemize}
 \item[(A3')] Any two germs of Weyl chambers are contained in a common apartment. 
\end{itemize}
\end{prop}
\begin{proof}
Let $S$ and $T$ be Weyl chambers based at $x$ and  $y$, respectively. 
By (A3) there exists a Weyl chamber $C$ based at $x$ containing $y$. Lemma~\ref{lem_tec1} implies that there exists an apartment $A$ containing $C$ and $\mu\define\Delta_xS$. But then there exists an $y$-based Weyl chamber $D$ in $A$ containing $\mu$. Applying \ref{lem_tec1} again, we obtain an apartment $A'$ containing $D$ and a germ of $T$ at $y$ and hence containing  $\mu=\Delta_xS$ and $\Delta_yT$. 
\end{proof}

\section{Retractions based at germs}\label{Sec_retractions}

Let throughout this section $(X,\App)$ be a pair satisfying axioms (A1), (A2) and (A3'') and fix an apartment $A$ in $X$ with chart $f \in \App$.

\begin{definition}\label{Def_vertexRetraction}
\index{vertex retraction}
Let $\mu$ be a germ of a Weyl chamber and $y$ a point in $X$, then, by (A3'),  there exists a chart $g\in\App$ such that $y$ and $\mu$ are contained in $g(\MS)$. 
By axiom (A2) there exists $w\in\aW$ such that $g\vert_{g^{-1}(f(\MS))}=(f\circ w)\vert_{g^{-1}(f(\MS))}$.
Hence we can define
$$
r_{A,\mu}(y) = (f\circ w\circ g^{-1} )(y).
$$
The map $r_{A,\mu}$ is called \emph{retraction onto $A$ centered at $\mu$}.
\end{definition}

\begin{prop}\label{Prop_r}
Fix an apartment $A$ of $X$ and let $\mu$ be a germ of a Weyl chamber in $A$. Then the following hold:
\begin{enumerate}
  \item The map $r_{A,\mu}$ is well defined.
  \item The restriction of the retraction $r_{A,\mu}$ to an apartment $A'$ containing $\mu$ is an isomorphism onto $A$.
\end{enumerate}
\end{prop}
\begin{proof}
The second assertions is clear by definition.
To prove the first let $y$ be a point in $X$ assume that $A_i\define f_i(\MS)$, $i=1,2$ are two apartments both containing $\mu$ and $y$. We let $w_i$ be the element of $\aW$ appearing in the definition of $r_{A,\mu}(y)$ with respect to $f_i$. It suffices to prove
\begin{equation}\label{Equ_tec28}
f\circ w_1\circ f_1^{-1}(y)=f\circ w_2\circ f_2^{-1}(y).
\end{equation}
By assumption the germ $\mu$ is contained in $A_1\cap A_2$ hence there exists by $(A2)$ an element $w_{12}\in\aW$ such that 
$$ 
f_2\circ w_{12}\;\vert_{f_1^{-1}(f_2(\MS))} = f_1 \;\vert_{f_1^{-1}(f_2(\MS))} .
$$
Since $y\in A_1\cap A_2$, we have 
\begin{equation}\label{Equ_tec29}
f\circ w_2\circ f_2^{-1} 
	= f\circ w_2\circ f_2^{-1}(f_2\circ w_{12}(f_1^{-1}(y)))
	= f\circ w_2 w_{12}(f_1^{-1}(y)).
\end{equation}
There are unique Weyl chambers $S_1$ and $S_2$ contained in $A_1$ and $A_2$, respectively, satisfying the property that $\Delta_xS_i=\mu$, $i=1,2$. Since equation (\ref{Equ_tec29}) is true for all $y\in A_1\cap A_2$, it is in particular true for the intersection $C$ of the Weyl chambers $S_1$ and $S_2$. Therefore 
$$
f\circ w_1\circ f_1^{-1}(C) = f\circ w_2\circ w_{12} \circ f_1^{-1}(C)
$$
and hence $w_2 w_{12}=w_1.$ Combining this with (\ref{Equ_tec29}) yields equation (\ref{Equ_tec28}).
\end{proof}

\section{Finite covering property}\label{Sec_FC1}

Assume that the pair $(X,\App)$ satisfies axioms (A1) to (A3) and has properties (GG) and (CO). Recall that we did prove in the previous section, that (A3') follows from this. Alternatively we may assume in place of (GG) axiom (A3').

\begin{lemma}\label{Lem_cover}
Given an apartment $A$ and a point $z$ in $X$. Then $A$ is contained in the (finite) union of all $z$-based Weyl chambers which are parallel to a Weyl chamber in $A$.
\end{lemma}
\begin{proof}
In case $z$ is contained in $A$ this is obvious. Hence we assume that $z$ is not contained in $A$. For all $p\in A$ there exists, by (A3), an apartment $A'$ containing $z$ and $p$. 
Let $S_+\subset A'$ be a $p$-based Weyl chamber containing $z$.  We denote by $\sigma_+$ its germ at $p$. There exists a $p$-based Weyl chamber $S_-$ in $A$ such that its germ $\sigma_-$ is opposite $\sigma_+$ at $p$. By property (CO) the Weyl chambers $S_-$ and $S_+$ are contained in a common apartment $A''$. 
Let $T$ be the unique $z$-based translate of $S_-$ in $A''$. Since $z\in S_+$ and $\sigma_+$ and $\sigma_-$ are opposite we have that $S_-\subset T$. In particular the point $p$ is contained in $T$. 
The fact that there are only finitely many chambers in $\partial A$ completes the proof. 
\end{proof}

\begin{prop}\label{Prop_FC'}
Let  $x$ and $y$ be points in $X$ and let $A$ be an apartment containing $X$ and $y$. For all $z\in X$ the following is true:
\begin{itemize}[label={(FC'')}, leftmargin=*]
  \item[$\mathrm{(FC'')}$] The segment $\seg_A(x,y)=\seg(x,y)\cap A$ of $x$ and $y$ is contained in a finite union of Weyl chambers based at $z$.
\end{itemize}
Furthermore, is $\mu$ a $z$-based germ of a Weyl chamber, then $\seg(x,y)$ is contained in a finite union of apartments containing $\mu$.
\end{prop}
\begin{proof}
Let $I$ be a (finite) index set of the $z$-based Weyl chambers $S_i$ with equivalence class in $\partial A$. Then, by Lemma~\ref{Lem_cover}, we may conclude that 
$\seg_A(x,y)\subset A\subset \bigcup_{i\in I} S_i$.
We fix $i$ and deduce from (GG) (or property (A3') instead) that there is an apartment $\widetilde{A}_i$ containing $\mu$ and $\Delta_zS_i$. Let $S_i^{op}$ be a Weyl chamber in $\widetilde{A}_i$ whose germ is opposite $\Delta_zS_i$. Then property (CO) implies that there is a unique apartment $A_i$ containing the union of $S_i$ and $S_i^{op}$. Hence $A$ and therefore $\seg_A(x,y)$ is contained in the finite union $\bigcup_{i\in I} A_i$. 
Hence the proposition. 
\end{proof}

\section{Verifying (A5)}\label{Sec_A5}

Assume that $(X,\App)$ is a pair satisfying axioms (A1), (A2), (A3'') and property (FC''). Observe that this is in particular satisfied under the hypotheses of Section~\ref{Sec_FC1} and that these conditions suffice to define retractions centered at germs,  as we did in Section~\ref{Sec_retractions}.

\begin{prop}\label{Prop_retraction}
For all apartments $A$ and germs $\mu$ of Weyl chambers contained in $A$ the retraction $r_{A,\mu}$, as defined in \ref{Def_vertexRetraction}, is distance non-increasing. In particular we conclude that the pair $(X,\App)$ satisfies axiom $(A5)$.
\end{prop}
\begin{proof}
Let $x$ and $y$ be points in an apartment $B$ of $X$. By (FC'') there exists a finite collection of apartments $A_0,\ldots, A_n$ each containing $\mu$ such that the union contains the segment of $x$ and $y$ in $B$. Let these apartments be enumerated such that $A_i\cap A_{i+1} \neq \emptyset$ for all $i=0,\ldots, n-1$. 
Observe that one can find a finite sequence of points $x_i$, $i=0,\ldots,n$ with $x_0=0$ and $x_n=y$ such that 
$$
d(x,y)=\sum_{i=0}^{n-1} d(x_i,x_{i+1})
$$
and such that $A_i$ contains $x_i$ and $x_{i+1}$.
Note further that for all $i$ the restriction of $r_{A,\mu}$ to $A_i$ is an isomorphism onto $A$. Hence the distance $d(x_i, x_{i+1})$ of $x_i$ and $x_{i+1}$ is equal to $d(\rho(x_i), \rho(x_{i+1}))$ for all $i\neq N$. Since the metric $d$ satisfies the triangle inequality on each apartment we have that $d(r(x),r(y))\leq \sum_{i=0}^{n-1} d(r(x_i), r(x_{i+1})) = \sum_{i=0}^{n-1} d(x_i, x_{i+1})= d(x,y)$.
\end{proof}

\section{Again: local structure}\label{Sec_lp}

Assume that $(X,\App)$ is a pair satisfying axioms (A1), (A2), (A3') and condition (CO). 
Under these assumptions Sections~\ref{Sec_retractions}, \ref{Sec_FC1} and \ref{Sec_A5} imply the existence of a distance diminishing retraction based at a germ of a Weyl chamber. That is (A5) holds. 

Alternatively we could assume that axioms (A1), (A2), (A3'), (A5)  are satisfied and that condition (CO) holds. These are precisely the properties needed in the present section.
Notice that the proof of Proposition~\ref{Prop_tec16} uses (A3') in its full power and that this axiom might therefore not be weakened to (A3'').

\begin{prop}\label{Prop_1.19Parreau}\label{Prop_tec16a}
Let $S$ be a Weyl chamber and $\mu$ a germ of another Weyl chamber, then there exists an apartment containing $\mu$ and a sub-Weyl chamber of $S$.
\end{prop}
\begin{proof}
Let $x$ be the base point of $S$ and let $\mu$ be based at $y$. Choose an apartment $A$ containing $S$ and let $z$ be a point in $S$. Denote by $S^+$ the sub-Weyl chamber of $S$ based at $z$ and refer to the $z$-based Weyl chamber in $A$ which is opposite $S^+$ at $z$ by $S^-$. Let further $r$ stand for the retraction onto $A$ centered at the germ of $S^+$ at $z$, which exists by (A5). For some $\varepsilon \geq 0$ the ball $B$ of radius $d(x,y)+\varepsilon$ around $x$ contains the image $r(\mu)$, since $r$ is distance diminishing. One can choose $z$ such that $B$ is contained in $S^-$.

By (A3') there exists an apartment $\hat{A}$ containing $\mu$ and $\Delta_zS^+$. We denote by $\hat{S}^-$ the unique $z$-based Weyl chamber in $\hat{A}$ whose germ at $z$ is opposite $\Delta_zS^+$. By construction $r$ maps $\hat{S}^-$ onto $S^-$. The Weyl chambers $S^+$ and $\hat{S}^-$ are opposite at $z$ and are therefore, by  property (CO),  contained in a common apartment. 
\end{proof}

\begin{corollary}\label{Cor_GGa}
Under the hypothesis of this section property (GG) holds for $X$. 
\end{corollary}
\begin{proof}
Let $S$ and $T$ be Weyl chambers both based at $x$. By Proposition~\ref{Prop_tec16a} there exists an apartment $A$ of $X$ containing $S$ and a germ of $T$ at $x$. Therefore $\Delta_xS$ and $\Delta_xT$ are both contained in the apartment $A$.
\end{proof}

For a proof of the next proposition compare p.13 in \cite{Parreau}.
\begin{prop}\label{Prop_parreau1.15}
Given two Weyl chambers $S$ and $T$ both based at $x$. Then there exists an apartment containing $S$ and a germ of $T$ at $x$. 
\end{prop}

\section{Verifying (A4)}\label{Sec_A4}

Assume that $(X,\App)$ is a pair satisfying axioms (A1) to (A3) and properties (GG) and (CO). Recall that we prove in Section~\ref{Sec_A3'} that the stronger axiom (A3') is then satisfied and that therefore the assertions of Section~\ref{Sec_lp} hold.
Alternatively we may assume that (A1), (A2), (A3') and (CO) are satisfied, which themselves imply property (GG). 
In Section~\ref{Sec_A3'} we did prove that these assumptions are enough to conclude that the germs at a given vertex form a spherical building. 

\begin{prop}
The pair $(X,\App)$ satisfies (A4). 
\end{prop}
\begin{proof}
Let $S$ and $T$ be two Weyl chambers in $X$. We will show that passing to sub-Weyl chambers we will find an apartment containing both. 

Given a point $x\in T$ we denote by $S_x$, respectively $T_x$, the unique $x$-based Weyl chambers parallel to $S$, respectively $T$. We denote by $\delta(x)$ the length of a minimal gallery from $\Delta_xS$ to $\Delta_xT$ in the spherical building $\Delta_xX$. 
Since the number of possible values for $\delta(x)$ is finite we may without loss of generality (by choosing different sub-Weyl chambers of $C'$ if necessary) assume  that $x$ is chosen such that $\delta(x)$ is maximal.

Now replace $S$ by $S_x$ and $T$ by $T_x$ where $x$ is such that $\delta(x)$ is maximal. Now in particular both $S$ and $T$ are based at $x$.
We let $A$ be an apartment containing $T$ and a germ of $S$ at $x$, which exists by Proposition~\ref{Prop_parreau1.15}, and we denote by $S'$ the $x$-based Weyl chamber in $A$ which is opposite $S$ at $x$. Property (CO) implies that there is an apartment $A'$ containing $S$ and $S'$. By (A2) the intersection $A\cap T$ is a convex subset of $T$. 
Let $z$ be a point in this intersection. The unique $z$-based sub-Weyl chambers $S_z$ of $S$ and $S_z''$ of $S''$ are both contained in $A'$. 
By construction the length of a minimal gallery from $\Delta_zS_z$ to $\Delta_zT_z$ is not greater than $\delta(x)$. On the other hand, since $T$ and $S'$ are both contained in the apartment $A$, we can conclude
$$
\delta_z(T_z, S_z') = \delta_x(T, S') = d-\delta_x(S,T) = d-\delta(x)
$$
where $d$ is the diameter of an apartment of $\Delta_xX$, that is the diameter of the spherical Coxeter complex associated to the underlying root system $\RS$. The function $\delta_x$ assigns to two $x$-based Weyl chambers the length of a minimal gallery connecting their germs in $\Delta_xX$.

The germ $\Delta_zT_z$ lies on a minimal gallery in connecting the opposite germs $\Delta_zS_z$ and $\Delta_zS_z'$. Such a  minimal gallery is contained in the unique apartment containing $\Delta_zS_z$ and $\Delta_zS_z'$, which is $\Delta_zA'$. Therefore $\Delta_zT_z$ is contained in $\Delta_zA'$ as well. 
This allows us to conclude that $A'\cap T$ contains a germ of $T_z$. One can observe that $A'\cap T$ is a convex subset of $T$ containing $x$ which is open relative to $T_z$. Hence the Weyl chamber $T$ is contained in $A'$. Thus (A4) follows. 
\end{proof}

\section{Exchange condition}\label{Sec_A6'}

The following \emph{exchange condition}, abbreviated by (EC) and introduced by Bennett in \cite{Bennett2}, is equivalent to (A6) assuming that axioms (A1) to (A5) hold. Compare \cite{Bennett2} for a proof of this fact.

\begin{itemize}
 \item[(EC)] Given two apartments $A$ and $B$ intersecting in a half-apartment $M$ with boundary wall $H$, then $(A\oplus B)\cup H$ is also an apartment, where $\oplus$ denotes the symmetric difference. 
\end{itemize}

We may restate condition (EC) as follows:
Given charts $f_1, f_2$ such that $f_1(\MS)\cap f_2(\MS)\ddefine M$ is a half apartment, then there exists a chart $f_3$ such that $f_3(\MS)\cap f_i(\MS)$ is a half-apartment for $i=1,2$. Moreover $f_3(\MS)$ is the symmetric difference of $f_1(\MS)$ and $f_2(\MS)$ together with the boundary wall of $M$.

\begin{prop}
Assume that $(X,\App)$ is a pair satisfying axioms (A1) to (A3) and property (CO) and require that the germs at each vertex form a spherical building. (This is true if for example in addition (GG) holds.) 
Then (EC) is satisfied.
\end{prop}
\begin{proof}
Let $A$ and $B$ be apartments intersecting in an half-apartment $M$. Let $x$ be a point contained in the bounding wall $H$ of $M$. By assumption $\Delta_xX$ is a spherical building. Therefore the union of $\Delta_x(A\setminus M)$,  $\Delta_x(B\setminus M)$ and $\Delta_xH$ is an apartment in $\Delta_xX$, which we denote by $\Delta_xA'$. 

We choose two opposite germs $\mu$ and $\sigma$ at $x$ which are contained in $\Delta_x(A\setminus M)$ and $\Delta_x(B\setminus M)$, respectively. Let $T$ be the unique Weyl chamber in $A$ having germ $\mu$ and let $S$ be the unique Weyl chamber in $B$ with germ $\sigma$. By construction an condition (CO) the Weyl chambers $S$ and $T$ are contained in a common apartment $A''$. Since two opposite Weyl chambers contained in the same apartment determine this apartment uniquely we can conclude that  $\Delta_xA''=\Delta_xA'$. We conclude that $A''\cap ((A\oplus B)\cup H)$ contains $S$, $T$ and $\Delta_xA'$. Axioms (A2) says that apartments intersect in convex sets. Therefore $A''\cap (B\setminus M) = B\setminus M$ and $A''\cap (A\setminus M) = A\setminus M)$ which implies that $A''\cap ((A\oplus B)\cup H) = A''$.
\end{proof}

\phantomsection
\renewcommand{\refname}{Bibliography}
\bibliography{literaturliste}
\bibliographystyle{alpha}

\end{document}